\documentclass[12pt]{article}
\usepackage{amsmath, amsthm, amssymb}
\usepackage{hyperref}
\usepackage{verbatim}
\usepackage[top=1.0in, bottom=1.0in, left=1.0in, right=1.0in]{geometry}

\usepackage{sectsty}
\allsectionsfont{\sffamily}

\makeatletter
\newtheorem*{rep@theorem}{\rep@title}
\newcommand{\newreptheorem}[2]{
\newenvironment{rep#1}[1]{
 \def\rep@title{#2 \ref{##1}}
 \begin{rep@theorem}}
 {\end{rep@theorem}}}
\makeatother

\theoremstyle{plain}
\newtheorem{thm}{Theorem}[section]
\newreptheorem{thm}{Theorem}

\newreptheorem{prop}{Proposition}
\newtheorem{lem}[thm]{Lemma}
\newreptheorem{lem}{Lemma}

\newreptheorem{conjecture}{Conjecture}
\newtheorem{cor}[thm]{Corollary}
\newreptheorem{cor}{Corollary}

\theoremstyle{definition}
\newtheorem{defn}{Definition}
\theoremstyle{remark}

\newcommand{\fancy}[1]{\mathcal{#1}}

\newcommand{\IN}{\mathbb{N}}

\newcommand{\CC}{\fancy{C}}
\newcommand{\D}{\fancy{D}}
\newcommand{\T}{\fancy{T}}
\newcommand{\B}{\fancy{B}}
\renewcommand{\L}{\fancy{L}}
\newcommand{\HH}{\fancy{H}}

\newcommand{\set}[1]{\left\{ #1 \right\}}
\newcommand{\setb}[3]{\left\{ #1 \in #2 : #3 \right\}}

\newcommand{\card}[1]{\left|#1\right|}
\newcommand{\size}[1]{\left\Vert#1\right\Vert}
\newcommand{\ceil}[1]{\left\lceil#1\right\rceil}

\newcommand{\func}[3]{#1\colon #2 \rightarrow #3}

\newcommand{\irange}[1]{\left[#1\right]}

\newcommand{\parens}[1]{\left( #1 \right)}
\newcommand{\brackets}[1]{\left[ #1 \right]}
\newcommand{\DefinedAs}{\mathrel{\mathop:}=}

\newcommand{\AT}{\operatorname{AT}}
\newcommand{\col}{\operatorname{col}}
\newcommand{\ch}{\operatorname{ch}}
\newcommand{\type}{\operatorname{type}}
\newcommand{\nonsep}{\bar{S}}

\newcommand\restr[2]{{% we make the whole thing an ordinary symbol
  \left.\kern-\nulldelimiterspace % automatically resize the bar with \right
  #1 % the function
  \vphantom{\big|} % pretend it's a little taller at normal size
  \right|_{#2} % this is the delimiter
  }}

\def\D{\fancy{D}}

\newcommand{\case}[2]{{\bf Case #1.}~{\it #2}~~}

\title{Improved lower bounds on the number of edges in list critical and online list critical graphs}
\author{H.\,A.~Kierstead and Landon Rabern\thanks{School of Mathematical and Statistical Sciences, Arizona State University}}
\date{\today}

\begin{document}
\maketitle

\begin{abstract}
We prove that every $k$-list-critical graph ($k \ge 7$) on $n \ge k+2$ vertices has at least $\frac12 \left(k-1 + \frac{k-3}{(k-c)(k-1) + k-3}\right)n$ edges where $c = (k-3)\left(\frac12 - \frac{1}{(k-1)(k-2)}\right)$.  This improves the bound established by Kostochka and Stiebitz \cite{kostochkastiebitzedgesincriticalgraph}.  The same bound holds for online $k$-list-critical graphs, improving the bound established by Riasat and Schauz \cite{riasat2012critically}.  Both bounds follow from a more general result stating that either a graph has many edges or it has an Alon-Tarsi orientable induced subgraph satisfying a certain degree condition.  
\end{abstract}

\section{Introduction}
A \emph{$k$-coloring} of a graph $G$ is a function $\func{\pi}{V(G)}{\irange{k}}$ such that $\pi(x) \ne \pi(y)$ for each $xy \in E(G)$.  The least $k$ for which $G$ has a $k$-coloring is the \emph{chromatic number} $\chi(G)$ of $G$. We say that $G$ is \emph{$k$-chromatic} when $\chi(G) = k$.  A graph $G$ is \emph{$k$-critical} if $G$ is not $(k-1)$-colorable, but every proper subgraph of $G$ is $(k-1)$-colorable. A $k$-critical graph $G$ is $k$-chromatic since for any vertex $v$, a $(k-1)$-coloring of $G-v$ extends to a $k$-coloring of $G$ by giving $v$ a new color.  If $G$ is $k$-chromatic, then any minimal $k$-chromatic subgraph of $G$ is $k$-critical.  In this way, many questions about $k$-chromatic graphs can be reduced to questions about $k$-critical graphs which have more structure.  The study of critical graphs was initiated by Dirac \cite{dirac1951note} in 1951.  It is easy to see that a $k$-critical graph $G$ must have minimum degree at least $k-1$ and hence $2\size{G} \ge (k-1)\card{G}$.  The problem of determining the minimum number of edges in a $k$-critical graph has a long history. First, in 1957, Dirac \cite{dirac1957theorem} generalized Brooks' theorem \cite{brooks1941colouring} by showing that any $k$-critical graph $G$ with $k \ge 4$ and $\card{G} \ge k+2$ must satisfy 

\[2\size{G} \ge (k-1)\card{G} + k-3.\]

In 1963, this bound was improved for large $\card{G}$ by Gallai \cite{gallai1963kritische}.  Put 
\[g_k(n, c) \DefinedAs \parens{k-1 + \frac{k-3}{(k-c)(k-1) + k-3}}n.\]
Gallai showed that every $k$-critical graph $G$ with $k \ge 4$ and $\card{G} \ge k+2$ satisfies $2\size{G} \ge g_k(\card{G}, 0)$.  In 1997, Krivelevich \cite{krivelevich1997minimal} improved Gallai's bound by replacing $g_k(\card{G}, 0)$ with $g_k(\card{G}, 2)$.  Then, in 2003, Kostochka and Stiebitz \cite{kostochkastiebitzedgesincriticalgraph} improved this by showing that a $k$-critical graph with $k \ge 6$ and $\card{G} \ge k+2$ must satisfy $2\size{G} \ge g_k(\card{G}, (k-5)\alpha_k)$ where

\[\alpha_k \DefinedAs \frac12 - \frac{1}{(k-1)(k-2)}.\]

Table \ref{tab:1} gives the values of these bounds for small $k$.  In 2012, Kostochka and Yancey \cite{kostochkayancey2012ore} achieved a drastic improvement by showing that every $k$-critical graph $G$ with $k \ge 4$ must satisfy

\[\size{G} \ge \ceil{\frac{(k+1)(k-2)\card{G} - k(k-3)}{2(k-1)}}.\]%HK Combined paragraphs
Moreover, they show that their bound is tight for $k=4$ and $n \ge 6$ as well as for infinitely many values of $\card{G}$ for any $k \ge 5$.  This bound has many interesting coloring applications such as a very short proof of Gr\"otsch's theorem on the $3$-colorability of triangle-free planar graphs \cite{kostochka2012oregrotsch} and short proofs of the results on coloring with respect to Ore degree in \cite{kierstead2009ore, rabern2010a, krs_one}.  %HK Broke paragraph. Expand table and add reference here???

Given the applications to coloring theory, it makes sense to investigate the same problem for more general types of coloring.  In this article, we obtain improved lower bounds on the number of edges for both the list coloring and online list coloring problems.  To state our results we need some definitions.

\emph{List coloring} was introduced by Vizing \cite{vizing1976} and independently Erd\H{o}s, Rubin and Taylor \cite{erdos1979choosability}.  Let $G$ be a graph. A list assignment on $G$ is a function $L$ from $V(G)$ to the subsets of $\IN$.   A graph $G$ is \emph{$L$-colorable} if there is $\func{\pi}{V(G)}{\IN}$ such that $\pi(v) \in L(v)$ for each $v \in V(G)$ and $\pi(x) \ne \pi(y)$ for each $xy \in E(G)$.   A graph $G$ is \emph{$L$-critical} if $G$ is not $L$-colorable, but every proper subgraph $H$ of $G$ is $\restr{L}{V(H)}$-colorable. For $\func{f}{V(G)}{\IN}$, a list assignment $L$ is an \emph{$f$-assignment} if $\card{L(v)} = f(v)$ for each $v \in V(G)$.  If $f(v) = k$ for all $v \in V(G)$, then we also call an $f$-assignment a $k$-assignment.  We say that $G$ is \emph{$f$-choosable} if $G$ is $L$-colorable for every $f$-assignment $L$.  
%HK We use "k-list-critical" in the table but do not define it.  I think we should have defined it, but if we do we should go through the paper and use it starting with the first sentence after the definition.  Also we should define and use "online k-list-critical 
We say that $G$ is $k$-list-critical if $G$ is $L$-critical for some $k$-list assignment $L$. %HK
The best, known-lower bound on the number of edges in a $k$-list-critical graph, was given by Kostochka and Stiebitz \cite{kostochkastiebitzedgesincriticalgraph} in 2003. %HK edits
It states that for $k \ge 9$ and every graph $G \ne K_k$ if $G$ is a $k$-list-critical graph, then $2\size{G} \ge g_k(\card{G}, \frac13 (k-4)\alpha_k)$.  We improve their bound to $2\size{G} \ge g_k(\card{G}, (k-3)\alpha_k)$ for $k\ge7$ (see Table \ref{tab:1}).

\emph{Online list coloring} was independently introduced by Zhu \cite{zhu2009online} and Schauz \cite{schauz2009mr} (Schauz called it \emph{paintability}). Let $G$ be a graph and $\func{f}{V(G)}{\IN}$.  We say that $G$ is \emph{online $f$-choosable} if $f(v) \ge 1$ for all $v \in V(G)$ and for every $S \subseteq V(G)$ there is an independent set $I \subseteq S$ such that $G-I$ is online $f'$-choosable where $f'(v) \DefinedAs f(v)$ for $v \in V(G) - S$ and $f'(v) \DefinedAs f(v) - 1$ for $v \in S - I$. %HK Added sentence.
Observe that if a graph is online $f$-choosable then it is $f$-choosable. 
 When $f(v) \DefinedAs k-1$ for all $v \in V(G)$, we say that $G$ is \emph{online $k$-list-critical} if $G$ is not online $f$-choosable, 
  but every proper subgraph $H$ of $G$ is online $\restr{f}{V(H)}$-choosable.  In 2012, Riasat and Schauz \cite{riasat2012critically} showed that Gallai's bound  $2\size{G} \ge g_k(\card{G}, 0)$ holds for online $k$-list-critical graphs.  We improve this for $k \ge 7$ %HK changed 8 to 7
 by proving the same bound as we have for list coloring: $2\size{G} \ge g_k(\card{G}, (k-3)\alpha_k)$.

Our main theorem shows that a graph either has many edges or an induced subgraph which has a certain kind of good orientation.  To describe these good orientations we need a few definitions. A subgraph $H$ of a directed multigraph $D$ is called \emph{Eulerian} if $d^-_H(v) = d^+_H(v)$ for every $v \in V(H)$.  We call $H$ \emph{even} if $\size{H}$ is even and \emph{odd} otherwise.  Let $EE(D)$ be the number of even, spanning, Eulerian subgraphs of $D$ and $EO(D)$ the number of odd, spanning, Eulerian subgraphs of $D$.  Note that the edgeless subgraph of $D$ is even and hence we always have $EE(D) > 0$.

Let $G$ be a graph and $\func{f}{V(G)}{\IN}$.  We say that $G$ is \emph{$f$-Alon-Tarsi} (for brevity, \emph{$f$-AT}) if $G$ has an orientation $D$ where $f(v) \ge d_{D}^+(v) + 1$ for all $v \in V(D)$ and $EE(D) \ne EO(D)$. One simple way to achieve $EE(D) \ne EO(D)$ is to have $D$ be acyclic since then we have $EE(D) = 1$ and $EO(D) = 0$.  In this case, ordering the vertices so that all edges point the same direction and coloring greedily shows that $G$ is $f$-choosable. If we require $f$ to be constant, we get the familiar \emph{coloring number} $\col(G)$; that is, $\col(G)$ is the smallest $k$ for which $G$ has an acyclic orientation $D$ with $k \ge d_{D}^+(v) + 1$ for all $v \in V(D)$.  Alon and Tarsi \cite{Alon1992125} generalized from the acyclic case to arbitrary $f$-AT orientations.

\begin{lem}\label{AlonTarsi}
If a graph $G$ is $f$-AT for $\func{f}{V(G)}{\IN}$, then $G$ is $f$-choosable.
\end{lem}

\noindent Schauz \cite{schauz2010flexible} extended this result to online $f$-choosability.

\begin{lem}\label{Schauz}
If a graph $G$ is $f$-AT for $\func{f}{V(G)}{\IN}$, then $G$ is online $f$-choosable.
\end{lem}

For a graph $G$, we define $\func{d_0}{V(G)}{\IN}$ by $d_0(v) \DefinedAs d_G(v)$.  The $d_0$-choosable graphs were first characterized by Borodin \cite{borodin1977criterion} and independently by Erd\H{o}s, Rubin and Taylor \cite{erdos1979choosability}.  The connected graphs which are not $d_0$-choosable are precisely the Gallai trees (connected graphs in which every block is complete or an odd cycle). The generalization to a characterization of $d_0$-AT graphs was first given in \cite{Hladky} by Hladk{\`y}, Kr{\'a}l and Schauz. 

We prove the following general theorem saying that either a graph has many edges or has an induced $f_H$-AT subgraph $H$ where $f_H$ basically gives the number of colors we would expect the vertices to have left in their lists after $\delta(G)$-coloring $G-H$.

\begin{defn}
	A graph $G$ is \emph{AT-reducible} to $H$ if $H$ is a nonempty induced subgraph of $G$ which is $f_H$-AT where $f_H(v) \DefinedAs \delta(G) + d_H(v) - d_G(v)$ for all $v \in V(H)$.  
	If $G$ is not AT-reducible to any nonempty induced subgraph, then it is \emph{AT-irreducible}.
\end{defn}

\begin{repthm}{EdgeBoundEuler}
	If $G$ is an AT-irreducible graph with $\delta(G) \ge 4$ and $\omega(G) \le \delta(G)$, then $2\size{G} \ge g_{\delta(G)+1}(\card{G}, c)$ where $c \DefinedAs (\delta(G)-2)\alpha_{\delta(G) + 1}$ when $\delta(G) \ge 6$ and $c \DefinedAs (\delta(G)-3)\alpha_{\delta(G) + 1}$ when $\delta(G) \in \set{4,5}$.
\end{repthm}

The \emph{Alon-Tarsi number} of a graph $AT(G)$ is the least $k$ such that $G$ is $f$-AT where $f(v) \DefinedAs k$ for all $v \in V(G)$. We have $\chi(G) \leq \ch(G) \leq \ch_{OL}(G) \leq AT(G) \leq \col(G)$.  We say that $G$ is $k$-AT-critical if $\AT(G) \ge k$ and $AT(H) < k$ for all proper induced subgraphs $H$ of $G$.  From Theorem \ref{EdgeBoundEuler} we can conclude the following.

\begin{repcor}{EdgeBoundAT}
For $k \ge 5$ and $G \ne K_k$ a $k$-AT-critical graph, we have $2\size{G} \ge g_k(\card{G}, c)$ where $c \DefinedAs (k-3)\alpha_k$ when $k \ge 7$ and $c \DefinedAs (k-4)\alpha_k$ when $k \in \set{5,6}$.
\end{repcor}

\begin{table} %HK I worked on the table.  
\begin{center}
\begin{tabular}{|c|c|c|c|c|c|c|}
\hline
&\multicolumn{4}{ |c| }{$k$-Critical $G$}&\multicolumn{2}{|c|}{$k$-ListCritical G}\\
\hline
& Gallai \cite{gallai1963kritische}
& Kriv \cite{krivelevich1997minimal}
& KS \cite{kostochkastiebitzedgesincriticalgraph}
& KY \cite{kostochkayancey2012ore}
 & KS \cite{kostochkastiebitzedgesincriticalgraph} & Here\\
$k$ & $d(G) \ge$ & $d(G) \ge$ & $d(G) \ge$ & $d(G) \ge$ & $d(G) \ge$ & $d(G) \ge$\\
\hline 
4 & 3.0769 &3.1429&---&3.3333& --- & --- \\
5 & 4.0909 &4.1429&---&4.5000& --- & \bf{4.0984}\\
6 & 5.0909 &5.1304&5.0976&5.6000& --- & \bf{5.1053}\\
7 & 6.0870 &6.1176&6.0990&6.6667& --- & \bf{6.1149}\\
8 & 7.0820 &7.1064&7.0980&7.7143& --- & \bf{7.1128}\\
9 & 8.0769 &8.0968&8.0959&8.7500& 8.0838 & \bf{8.1094}\\
10 & 9.0722 &9.0886&9.0932&9.7778& 9.0793 & \bf{9.1055}\\
%11 & 10.0678 & 10.0751 & \bf{10.1014}\\
%12 & 11.0638 & 11.0711 & \bf{11.0974}\\
%13 & 12.0602 & 12.0674 & \bf{12.0935}\\
%14 & 13.0570 & 13.0641 & \bf{13.0899}\\
15 & 14.0541 &14.0618&14.0785&14.8571& 14.0610 & \bf{14.0864}\\
%16 & 15.0514 & 15.0582 & \bf{15.0831}\\
%17 & 16.0490 & 16.0556 & \bf{16.0800}\\
%18 & 17.0467 & 17.0533 & \bf{17.0771}\\
%19 & 18.0447 & 18.0511 & \bf{18.0744}\\
20 & 19.0428 &19.0474&19.0666&19.8947& 19.0490 & \bf{19.0719}\\
\hline
\end{tabular}
\end{center}
\caption{History of lower bounds on the average degree $d(G)$ of $k$-critical and $k$-list-critical graphs $G$.}
\label{tab:1}
\end{table}

%\begin{table} %HK I shortented the table
%\begin{center}
%\begin{tabular}{|c|c|c|c|}
%\hline
%& Gallai \cite{gallai1963kritische} & KS \cite{kostochkastiebitzedgesincriticalgraph} & Here\\
%$k$ & $d(G) \ge$ & $d(G) \ge$ & $d(G) \ge$\\
%\hline 
%4 & 3.0769 & --- & --- \\
%5 & 4.0909 & --- & \bf{4.0984}\\
%6 & 5.0909 & --- & \bf{5.1053}\\
%7 & 6.0870 & --- & \bf{6.1149}\\
%8 & 7.0820 & --- & \bf{7.1128}\\
%9 & 8.0769 & 8.0838 & \bf{8.1094}\\
%10 & 9.0722 & 9.0793 & \bf{9.1055}\\
%%11 & 10.0678 & 10.0751 & \bf{10.1014}\\
%%12 & 11.0638 & 11.0711 & \bf{11.0974}\\
%%13 & 12.0602 & 12.0674 & \bf{12.0935}\\
%%14 & 13.0570 & 13.0641 & \bf{13.0899}\\
%15 & 14.0541 & 14.0610 & \bf{14.0864}\\
%%16 & 15.0514 & 15.0582 & \bf{15.0831}\\
%%17 & 16.0490 & 16.0556 & \bf{16.0800}\\
%%18 & 17.0467 & 17.0533 & \bf{17.0771}\\
%%19 & 18.0447 & 18.0511 & \bf{18.0744}\\
%20 & 19.0428 & 19.0490 & \bf{19.0719}\\
%\hline
%\end{tabular}
%\end{center}
%\caption{History of lower bounds on the average degree $d(G)$ of a $k$-list-critical graph $G$.}
%\label{tab:1}
%\end{table}

\noindent Similarly, applying Lemma \ref{AlonTarsi} gives the following.

\begin{repcor}{EdgeBound}
For $k \ge 5$ and $G \ne K_k$ a $k$-list-critical graph, we have  $2\size{G} \ge g_k(\card{G}, c)$ where $c \DefinedAs (k-3)\alpha_k$ when $k \ge 7$ and $c \DefinedAs (k-4)\alpha_k$ when $k \in \set{5,6}$.
\end{repcor}

This improves the bound given by Kostochka and Stiebitz in \cite{kostochkastiebitzedgesincriticalgraph}; for $k$-list-critical graphs, they have $2\size{G} \ge g_k(\card{G}, \frac13 (k-4)\alpha_k)$ for $k \ge 9$.  Now, applying Lemma \ref{Schauz} gives the following.

\begin{repcor}{EdgeBoundOnline}
For $k \ge 5$ and $G \ne K_k$ an online $k$-list-critical graph, we have  $2\size{G} \ge g_k(\card{G}, c)$ where $c \DefinedAs (k-3)\alpha_k$ when $k \ge 7$ and $c \DefinedAs (k-4)\alpha_k$ when $k \in \set{5,6}$.
\end{repcor}

\section{Critical graphs are AT-irreducible}
Instead of proving lower bounds on the number of edges in critical graphs directly, we prove our bound for AT-irreducible graphs and show that graphs that are critical with respect to choice number, online choice number and Alon-Tarsi number are all AT-irreducible.  In this section, we take on the easier task of proving that the various critical graphs are AT-irreducible.

\begin{lem}\label{ChoosingIrreducible}
	If $G$ is a $k$-list-critical graph, then $G$ is AT-irreducible.
\end{lem}
\begin{proof}
	Suppose $G$ is AT-reducible to $H$.  Let $L$ be a $(k-1)$-assignment on $G$ such that $G$ is $L$-critical.  Let $\pi$ be a coloring of $G-H$ from $L$ and let $L'$ be the list assignment on $H$ defined by $L'(v) \DefinedAs L(v) - \pi(N(v) \cap V(G-H))$ for $v \in V(H)$.  Then $\card{L'(v)} \ge \card{L(v)} - (d_G(v) - d_H(v)) = k - 1 + d_H(v) - d_G(v)$.  By Lemma \ref{AlonTarsi}, $H$ is $f_H$-choosable and hence $H$ is $L'$-colorable.  Therefore $G$ is $L$-colorable, a contradiction.
\end{proof}

For online list coloring, we use the following lemma from \cite{schauz2009mr} allowing us to patch together online list colorability of parts into online list colorability of the whole.

\begin{lem}\label{CutLemma}
	Let $G$ be a graph and $\func{f}{V(G)}{\IN}$.  If $H$ is an induced subgraph of $G$ such that $G-H$ is online $\restr{f}{V(G-H)}$-choosable and $H$ is online $f_H$-choosable where $f_H(v) \DefinedAs f(v) + d_H(v) - d_G(v)$, then $G$ is online $f$-choosable.
\end{lem}

\begin{lem}\label{OnlineChoosingIrreducible}
	If $G$ is an online $k$-list-critical graph, then $G$ is AT-irreducible.
\end{lem}
\begin{proof}
	Immediate from Lemma \ref{CutLemma} and Lemma \ref{Schauz}.	
\end{proof}

To prove that $k$-AT-critical graphs are AT-irreducible, we need a lemma that serves the same purpose as Lemma \ref{CutLemma} for orientations.

\begin{lem}\label{CutLemmaAT}
	Let $G$ be a graph and $\func{f}{V(G)}{\IN}$.  If $H$ is an induced subgraph of $G$ such that $G-H$ is  $\restr{f}{V(G-H)}$-AT and $H$ is $f_H$-AT where $f_H(v) \DefinedAs f(v) + d_H(v) - d_G(v)$, then $G$ is $f$-AT.
\end{lem}
\begin{proof}
	Take an orientation of $G-H$ demonstrating that it is $\restr{f}{V(G-H)}$-AT and an orientation of $H$ demonstrating that it is $f_H$-AT.  Now orient all the edges between $H$ and $G-H$ into $G-H$.  Call the resulting oriented graph $D$. Then $D$ satisfies the out degree requirements of being $f$-AT since the out degree of the vertices in $G-H$ haven't changed and the out degree of each $v \in V(H)$ has increased by $d_G(v) - d_H(v)$.  Since no directed cycle in $D$ has vertices in both $H$ and $D-H$, the Eulerian subgraphs of $D$ are just all pairings of Eulerian subgraphs of $H$ and $D-H$.  Therefore $EE(D) - EO(D) = EE(H)EE(D-H) + EO(H)EO(D-H) - (EE(H)EO(D-H) + EO(H)EE(D-H)) = (EE(H) - EO(H))(EE(D-H) - EO(D-H)) \ne 0$.  Hence $G$ is $f$-AT.
\end{proof}

\begin{lem}\label{ATNumberIrreducible}
	If $G$ is a $k$-AT-critical graph, then $G$ is AT-irreducible.
\end{lem}
\begin{proof}
	Immediate from Lemma \ref{CutLemmaAT}.
\end{proof}

\section{Extending Alon-Tarsi orientations}
In \cite{kostochkastiebitzedgesincriticalgraph} Kostochka and Stiebitz gave a method for extending list colorings into Gallai trees. We generalize these ideas in terms of extensions of orientations.  Let $\T_k$ be the Gallai trees with maximum degree at most $k-1$, excepting $K_k$. For a graph $G$, let $W^k(G)$ be the set of vertices of $G$ that are contained in some $K_{k-1}$ in $G$.

\begin{lem}\label{ConfigurationTypeTwoEuler}
Let $G$ be a multigraph without loops and $\func{f}{V(G)}{\IN}$. If there are $F \subseteq G$ and
$Y \subseteq V(G)$ such that:
\begin{enumerate}
\item any multiple edges in $G$ are contained in $G[Y]$; and
\item $f(v) \ge d_G(v)$ for all $v \in V(G) - Y$; and
\item $f(v) \ge d_{G[Y]}(v) + d_F(v) + 1$ for all $v \in Y$; and
\item For each component $T$ of $G-Y$ there are different $x_1, x_2 \in V(T)$ where $N_T[x_1] = N_T[x_2]$ and $T - \set{x_1, x_2}$ is connected such that either:
	\begin{enumerate}
	\item there are $x_1y_1, x_2y_2 \in E(F)$ where $y_1 \ne y_2$ and $N(x_i) \cap Y = \set{y_i}$ for $i \in \irange{2}$; or
	\item $\card{N(x_2) \cap Y} = 0$ and there is $x_1y_1 \in E(F)$ where $N(x_1) \cap Y = \set{y_1}$,
	\end{enumerate}
\end{enumerate}

\noindent then $G$ is $f$-AT.
\end{lem}
\begin{proof}
Suppose not and pick a counterexample $\parens{G, f, F, Y}$ minimizing $\card{G-Y}$. %HK Changed sentence to fix the typesetting
 If $\card{G-Y} = 0$, then $Y = V(G)$ and thus $f(v) \ge d_G(v) + 1$ for all $v \in V(G)$ by (3).  Pick an acyclic orientation $D$ of $G$.  Then $EE(D) = 1$, $EO(D) = 0$ and $d_D^+(v) \le d_G(v) \le f(v) - 1$ for all $v \in V(D)$. Hence $G$ is $f$-AT.  So, we must have $\card{G-Y} > 0$.  

Pick a component $T$ of $G - Y$ and pick $x_1, x_2 \in V(T)$ as guaranteed by (4). First, suppose (4a) holds.   Put $G' \DefinedAs (G - T) + y_1y_2$, $F' \DefinedAs F - T$, $Y' \DefinedAs Y$ and let $f'$ be $f$ restricted to $V(G')$.  Then $G'$ has an orientation $D'$ where $f'(v) \ge d_{D'}^+(v) + 1$ for all $v \in V(D')$ and $EE(D') \ne EO(D')$, for otherwise $\parens{G', f', F', Y'}$ would contradict minimality.  By symmetry we may assume that the new edge $y_1y_2$ is directed toward $y_2$.  Now we use the orientation of $D'$ to construct the desired orientation of $D$. First, we use the orientation on $D' - y_1y_2$ on $G-T$. Now, order the vertices of $T$ as $x_1, x_2, z_1, z_2, \ldots$ so that every vertex has at least one neighbor to the right.  Orient the edges of $T$ left-to-right in this ordering.  Finally, we use $y_1x_1$ and $x_2y_2$ and orient all other edges between $T$ and $G-T$ away from $T$.  Plainly, $f(v) \ge d_{D}^+(v) + 1$ for all $v \in V(D)$.  Since $y_1x_1$ is the only edge of $D$ going into $T$, any Eulerian subgraph of $D$ that contains a vertex of $T$ must contain $y_1x_1$.  So, any Eulerian subgraph of $D$ either contains (i) neither $y_1x_1$ nor $x_2y_2$, (ii) both $y_1x_1$ and $x_2y_2$, or (iii) $y_1x_1$ but not $x_2y_2$.  We first handle (i) and (ii) together.  Consider the function $h$ that maps an Eulerian subgraph $Q$ of $D'$ to an Eulerian subgraph $h(Q)$ of $D$ as follows.  If $Q$ does not contain $y_1y_2$, let $h(Q) = \iota(Q)$ where $\iota(Q)$ is the natural embedding of $D' - y_1y_2$ in $D$.  Otherwise, let $h(Q) = \iota(Q - y_1y_2) + \set{y_1x_1, x_1x_2, x_2y_2}$.  Then $h$ is a parity-preserving injection with image precisely the union of those Eulerian subgraphs of $D$ in (i) and (ii).  Hence if we can show that exactly half of the Eulerian subgraphs of $D$ in (iii) are even, we will conclude $EE(D) \ne EO(D)$, a contradiction.  To do so, consider an Eulerian subgraph $A$ of $D$ containing $y_1x_1$ and not $x_2y_2$. Since $x_1$ must have in-degree $1$ in $A$, it must also have out-degree $1$ in $A$.  We show that $A$ has a mate $A'$ of opposite parity.  Suppose $x_2 \not \in A$ and $x_1z_1 \in A$; then we make $A'$ by removing $x_1z_1$ from $A$ and adding $x_1x_2z_1$.  If $x_2 \in A$ and $x_1x_2z_1 \in A$, we make $A'$ by removing $x_1x_2z_1$ and adding $x_1z_1$. Hence exactly half of the Eulerian subgraphs of $D$ in (iii) are even and we conclude $EE(D) \ne EO(D)$, a contradiction.

Now suppose (4b) holds.  Put $G' \DefinedAs G - T$, $F' \DefinedAs F - T$, $Y' \DefinedAs Y$ and define $f'$ by $f'(v) = f(v)$ for all $v \in V(G'-y_1)$ and $f'(y_1) = f(y_1) - 1$.  Then $G'$ has an orientation $D'$ where $f'(v) \ge d_{D'}^+(v) + 1$ for all $v \in V(D')$ and $EE(D') \ne EO(D')$, for otherwise $\parens{G', f', F', Y'}$ would contradict minimality.  We orient $G - T$ according to $D$, orient $T$ as in the previous case, again use $y_1x_1$ and orient all other edges between $T$ and $G-T$ away from $T$.  Since we decreased $f'(y_1)$ by $1$, the extra out edge of $y_1$ is accounted for and we have $f(v) \ge d_{D}^+(v) + 1$ for all $v \in V(D)$.  Again any additional Eulerian subgraph must contain $y_1x_1$ and since $x_2$ has no neighbor in $G-T$ we can use $x_2$ as before to build a mate of opposite parity for any additional Eulerian subgraph.  Hence $EE(D) \ne EO(D)$ giving our final contradiction.
\end{proof}

\begin{lem}\label{ConfigurationTypeOneSingleEuler}
Let $r \ge 0$, $k \ge r + 4$ and $G \ne K_k$ be a graph with $x \in V(G)$ such that:
\begin{enumerate}
\item $G-x \in \T_k$; and
\item $d_G(x) \ge r + 2$; and
\item $\card{N(x) \cap W^k(G-x)} \ge 1$; and
\item $d_G(v) \leq k - 1$ for all $v \in V(G-x)$.
\end{enumerate}

\noindent Then $G$ is $f$-AT where $f(x) = d_G(x) - r$ and $f(v) = d_G(v)$ for all $v \in V(G - x)$.
\end{lem}
\begin{proof}
Suppose not and choose a counterexample minimizing $\card{G}$.  Let $Q$ be the %HK set of
 set of non-separating vertices in $G-x$. Suppose we have $y \in Q$ such that $G-y$ satisfies all the hypotheses of the theorem. Then minimality of $\card{G}$ shows that $G-y$ is $f'$-AT where $f'(v) \DefinedAs f(v) + d_{G-y}(v) - d_G(v)$ for $v \in V(G)$.  Create an orientation $D$ of $G$ from the orientation of $G-y$ by directing all edges incident to $y$ into $y$.  These new edges are on no cycle and thus the Eulerian subgraph counts did not change.  Also, we have increased the out degree of any vertex $v$ by at most $d_G(v) - d_{G-y}(v)$.  Hence $G$ is $f$-AT, a contradiction.  Therefore $G-y$ must fail some hypothesis for each $y \in Q$; note that it is only possible for $G-y$ to fail (2) or (3).

We show that $Q \subseteq N(x)$.  Suppose otherwise that we have $y \in Q - N(x)$.  
Since (2) is satisfied for $G-y$, (3) must fail and hence $y$ is contained in a $K_{k-1}$, call it $B$, in $G-x$ such that $N(x) \cap B \ne \emptyset$. Pick $z \in N(x) \cap B$. Since $d_G(z) \leq k-1$ we must have $N_{G-x}(z) \subseteq B$ and hence $z \in Q$.  Since $y \in Q$ and $G-x \in \T_k$, we must have $N_{G-x}(y) \subseteq B$. But then the conditions of Lemma \ref{ConfigurationTypeTwoEuler} are satisfied with $F \DefinedAs G[x, z]$ and $Y \DefinedAs \set{x}$ since $f(x) \ge d_G(x) - r \ge 2 = d_{G[Y]}(x) + d_F(x) + 1$. This is a contradiction and hence we must have $Q \subseteq N(x)$.

Now, by (3), $G-x$ has at least one $K_{k-1}$, call it $B$, such that $N(x) \cap V(B) \ne \emptyset$.  If $V(G-x) = B$, then $B = Q \subseteq N(x)$ and $G = K_k$, impossible.  Hence we may pick $y \in Q - B$.  Then $G-y$ satisfies (3) and hence must not satisfy (2).  We conclude that $d_G(x) = r+2$ and hence $\card{Q} \leq r+2$.  But $\card{Q} \ge \Delta(G-x) = k-1$ and hence $k \leq r + 3$, a contradiction.
\end{proof}

We will need to know what happens when we patch two $d_0$-choosable graphs together at a vertex.  To determine this we first need to understand the structure of $d_0$-choosable graphs.  The $d_0$-choosable graphs were first characterized by Borodin \cite{borodin1977criterion} and independently by Erd\H{o}s, Rubin and Taylor \cite{erdos1979choosability}.  The generalization to a characterization of $d_0$-AT graphs was first given in \cite{Hladky} by Hladk{\`y}, Kr{\'a}l and Schauz.  This generalization is easily derived from %HK too many follows
 the following lemma from \cite{erdos1979choosability} that is often referred to as ``Rubin's Block Theorem''. %HK reworded

\begin{lem}[Rubin \cite{erdos1979choosability}]\label{RubinBlock}
A $2$-connected graph is either complete, an odd cycle or contains an induced even cycle with at most one chord.
\end{lem}

\begin{lem}\label{d0Characterization}
For a connected graph $G$, the following are equivalent:
\begin{enumerate}
\item $G$ is not a Gallai tree,
\item $G$ contains an even cycle with at most one chord,
\item $G$ is $d_0$-choosable,
\item $G$ is $d_0$-AT,
\item $G$ has an orientation $D$ where $d_G(v) \ge d_{D}^+(v) + 1$ for all $v \in V(D)$, $EE(D) \in \set{2,3}$ and $EO(D) \in \set{0,1}$.
\end{enumerate}
\end{lem}
\begin{proof}
That (1), (2) and (3) are equivalent is the characterization of $d_0$-choosable graphs in \cite{borodin1977criterion} and \cite{erdos1979choosability}.  Since (5) implies (4) and (4) implies (3) it will suffice to show that (2) implies (5).  The proof we give of (5) is the same as in \cite{Hladky}. Suppose (2) holds and let $H$ be an induced even cycle with at most one chord in $G$.  Orient the even cycle in $H$ clockwise and the (possible) other edge arbitrarily.  Contract $H$ to a single vertex $x_H$ to form $H'$ and take a spanning tree $T$ of $H'$ with root $x_H$.  Orient the remaining edges in $G$ away from the root in this tree to get $D$.  Then every vertex has in degree at least $1$ in $D$ and hence $d_G(v) \ge d_{D}^+(v) + 1$ for all $v \in V(D)$.  Also, since the orientation of $D-H$ is acyclic, the only spanning Eulerian subgraphs of $D$ are the edgeless graph, the graph with just the edges from the even cycle in $H$ and possibly one other using the chord in $H$.  Hence $EE(D) \in \set{2,3}$ and $EO(D) \in \set{0,1}$, thus (5) holds. 
\end{proof}

\begin{lem}\label{CutvertexListPatchEuler}
If $\set{A, B}$ is a separation of $G$ such that $G[A]$ and $G[B]$ are connected $d_0$-AT graphs and $A \cap B = \set{x}$, then $G$ is $f$-AT where $f(v) = d_G(v)$ for all $v \in V(G) - x$ and $f(x) = d_G(x) - 1$.
\end{lem}
\begin{proof}
By Lemma \ref{d0Characterization} we may choose an orientation $D_A$ of $A$ with $d^+(v) < d(v)$ for all $v \in V(D_A)$ and $EE(D_A) \ne EO(D_A)$ and an orientation $D_B$ of $B$ with $d^+(v) < d(v)$ for all $v \in V(D_B)$ and $EE(D_B) \ne EO(D_B)$.  Together these give the desired orientation $D$ of $G$ since no cycle has vertices in both $A-x$ and $B-x$ and thus $EE(D) - EO(D) = EE(D_A)EE(D_B) + EO(D_A)EO(D_B) - (EE(D_A)EO(D_B) + EO(D_A)EE(D_B)) = (EE(D_A) - EO(D_A))(EE(D_B) - EO(D_B)) \ne 0$.
\end{proof}

Lemma \ref{ConfigurationTypeOneSingleEuler} restricts the interaction of a high vertex and a single low component.  Similarly to \cite{kostochkastiebitzedgesincriticalgraph} we'll use the following lemma to restrict a high vertex's interaction with two low components.

\begin{lem}\label{ConfigurationTypeOneDoubleEuler}
Let $k \ge 4$ and let $G$ be a graph with $x \in V(G)$ such that:
\begin{enumerate}
\item $G-x$ has two components $H_1, H_2 \in \T_k$; and
\item $\card{N(x) \cap V(H_i)} = 2$ for $i \in \irange{2}$; and
\item $\card{N(x) \cap W^k(H_i)} \ge 1$ for $i \in \irange{2}$.
\end{enumerate}

\noindent Then $G$ is $f$-AT where $f(x) = d_G(x) - 1$ and $f(v) = d_G(v)$ for all $v \in V(G - x)$.
\end{lem}
\begin{proof}
Using Lemma \ref{CutvertexListPatchEuler}, we just need to show that $Q_i \DefinedAs G\brackets{\set{x} \cup V(H_i)}$ is $d_0$-AT for $i \in \irange{2}$; that is show that $Q_i$ is not a Gallai tree.  
If $Q_i$ is a Gallai tree, then $x$'s two neighbors in $H_i$ must be in the same block in $H_i$ and this block must be a $K_{k-1}$, but this creates a diamond since $k \ge 4$, impossible.
\end{proof}

Combining Lemma \ref{ConfigurationTypeOneSingleEuler} and Lemma \ref{ConfigurationTypeOneDoubleEuler} gives the following.

\begin{lem}\label{ConfigurationTypeOneEuler}
Let $k \ge 5$ and let $G$ be a graph with $x \in V(G)$ such that:
\begin{enumerate}
\item $K_k \not \subseteq G$; and
\item $G-x$ has $t$ components $H_1, H_2, \ldots, H_t$, and all are in $\T_k$; and
\item $d_G(v) \leq k - 1$ for all $v \in V(G-x)$; and
\item $\card{N(x) \cap W^k(H_i)} \ge 1$ for $i \in \irange{t}$; and
\item $d_G(x) \ge t+2$.
\end{enumerate}

\noindent Then $G$ is $f$-AT where $f(x) = d_G(x) - 1$ and $f(v) = d_G(v)$ for all $v \in V(G - x)$.
\end{lem}
\begin{proof}
Since $d_G(x) \ge t+2$, either $x$ has $3$ neighbors in some $H_i$ or $x$ has two neighbors in each of $H_i, H_j$.  In either case, let $C_1, \ldots, C_q$ be the other components of $G-x$.  For each $i \in \irange{q}$, pick $z_i \in N(x) \cap V(C_i)$.  Then order the vertices of $C_i$ with $z_i$ first and orient all the edges in $C_i$ to the right with respect to this ordering.  Now orient all edges between $C_i$ and $G-C_i$ into $C_i$.  Note that each vertex in $C_i$ has in-degree at least one and no cycle passes through $C_i$. Hence we can complete the orientation using one of Lemma \ref{ConfigurationTypeOneSingleEuler} or Lemma \ref{ConfigurationTypeOneDoubleEuler} to get our desired orientation $D$ of $G$.
\end{proof}

To deal with more than one high vertex we need to define the following auxiliary bipartite graph.  For a graph $G$, $\set{X, Y}$ a partition of $V(G)$ and $k \ge 4$, let $\B_k(X, Y)$ be the bipartite graph with one part $Y$ and the other part the components of $G[X]$.  Put an edge between $y \in Y$ and a component $T$ of $G[X]$ iff $N(y) \cap W^k(T) \ne \emptyset$.  Lemma \ref{MultipleHighConfigurationEuler} gives the substantive improvement over \cite{kostochkastiebitzedgesincriticalgraph} on the lower bound on the number of edges in a list critical graph.  Before proceeding we need a lemma about orientations.  

Let $G=(V,E)$ be a multigraph. A function $A:V\rightarrow\wp(E)$ is called an \emph{incidence
	preference.} Set $d(v,A)=d_{G}(v,A)=|E(v)\cap A(v)|$. Call an edge $uv$ $A$-good
(or just good) if $uv\in A(u)\cap A(v)$, and let $A(G)$ be the set of good edges
of $G$. If $D$ is an orientation of $G$, set $d^{-}(v,A)=\card{\setb{(u,v)}{E(D)}{\set{u,v} \in A(v)}}$. 
\begin{lem}
	\label{InOrientationsIncidencePreference} Let $G$ be a graph with incidence preference $A$, $S\subseteq V(G)$
	and $\func{g}{S}{\IN}$. Then $G$ has an orientation such that $d^{-}(v,A)\ge g(v)$
	for all $v\in S$ iff for every $H\unlhd G[S]$
	\[\sum_{v\in V(H)}d(v,A)-|A(H)|\ge\sum_{v\in V(H)}g(v).\]
\end{lem}
\begin{proof}
	First, suppose $G$ has such an orientation $D$ with $d^{-}(v,A)\ge g(v)$ for all
	$v\in S$. Consider any $H\unlhd G[S]$. Then the second sum in \eqref{1-eq} equals
	$|\{uv\in E(D):v\in V(H)~\mbox{and}~uv\in A(v)\}|$, and the third sum equals $|\{uv\in E(G):v\in V(H)~\mbox{and}~uv\in A(v)\}|$.
	So 
\begin{equation}
	\sum_{v\in V(H)}g(v)\leq\sum_{v\in V(H)}d_{H}^{-}(v,A)\leq\sum_{v\in V(H)}d(v,A)-|A(H)|.\label{1-eq}
\end{equation}
	
	For the other direction, pick an orientation $D$ of $G$ minimizing
	\[
	\Theta\DefinedAs\sum_{v\in S}\max\set{0,g(v)-d^{-}(v,A)}.
	\]
	It suffices to show $\Theta=0$. If not then there is $x_{0}\in S$ with $d^{-}(x_{0})<g(x_{0})$.
	Put
	\[
	X\DefinedAs\{v\in V(G):(\exists P_{v}:=x_{0}x_{1}\dots x_{t}~\mbox{with}~v=x_{t})(\forall i\in[t])[v_{i-1}v_{i}\in E(D)\cap A(v_{i-1})]\}.
	\]
	Every $v\in X$ satisfies $d^{-}(v,A)\le g(v)$ for otherwise reversing all the edges
	on $P_{v}$ violates the minimality of $\Theta$. By definition, all edges $vw\in E(G)\cap A(v)$
	with $v\in X$ and $w\in G-X$ are directed into $X$, so with $H\DefinedAs G[X]$
	we have the contradiction
	\[
	\sum_{v\in X}d(v,A)-|A(H)|=\sum_{v\in V(H)}d^{-}(v,A)<g(x_{0})+\sum_{v\in V(H)-x_{0}}d^{-}(v,A)\leq\sum_{v\in V(H)}g(v).\qedhere
	\]
\end{proof}

For a graph $G$, let $\nonsep(G)$ be the subset of non-separating vertices of $G$.

\begin{lem}
	\label{MultipleHighConfigurationEuler} Let $k\ge7$ and let $G$ be a graph with
	$Y\subseteq V(G)$ such that: 
	\begin{enumerate}
		\item $K_{k}\not\subseteq G$; and 
		\item the components of $G-Y$ are in $\T_{k}$; and 
		\item $d_{G}(v)\leq k-1$ for all $v\in V(G-Y)$; and 
		\item with $\B\DefinedAs\B_{k}(V(G-Y),Y)$ we have $\delta(\B)\ge3$. 
	\end{enumerate}
	\noindent Then $G$ has an induced subgraph $G'$ that is $f$-AT where $f(y)=d_{G'}(y)-1$
	for $y\in Y$ and $f(v)=d_{G'}(v)$ for all $v\in V(G'-Y)$.\end{lem}
\begin{proof}
	\noindent Suppose not and pick a counterexample $G$ minimizing $\card{G}$. Note
	that $\size{w,Y}_G \le 1$ for every $w\in W^{k}(T)$, and if $\size{w,Y}_G = 1$
	then $w\in\nonsep(T)$; so if $y\in Y$ and $T$ is a component of $G-Y$ then $N(y)\cap W^{k}(T)\subseteq\nonsep(T)$.  
	By Lemma \ref{ConfigurationTypeOneEuler}, $\size{y,T}_G \le 2$ for each edge $yT$ of $\B$ since otherwise
	$G' = G[N_{\B}[y]]$ satisfies the conclusion of the lemma.
	Call an edge $yT$ of $\B$ \emph{heavy} if $\size{y,T}_G = 2$.
	Let $\HH$ be the set of heavy edges, and $H = \bigcup_{yT \in \HH} \setb{yx}{E(G)}{x \in V(T)}$. For $v\in S\subseteq V(\B)$,
	set $h(v)=\card{E_{\B}(v)\cap\HH}$ and $h(S)=\sum_{v\in S}h(v)$.
	By Lemma \ref{ConfigurationTypeOneEuler}, $h(y) \le 1$ for all $y \in Y$ since otherwise
	$G' = G[N_{\B}[y]]$ satisfies the conclusion of the lemma. 
	
	Suppose a component $T$ of $G-Y$ has an endblock $B$ with $B\ne K_{k-1}$ or $E(\nonsep(B),Y)=\emptyset$.
	Then $G'\DefinedAs G-\nonsep(B)$ still satisfies the hypotheses of the theorem since
	the degrees in $\B$ are not affected. Hence, by minimality of $\card{G}$, there
	is an induced subgraph $G''\subseteq G'$ that is $f$-AT where $f(y)=d_{G''}(y)-1$
	for $y\in Y$ and $f(v)=d_{G''}(v)$ for all $v\in V(G''-Y)$. But $G''$ is also
	an induced subgraph of $G$, a contradiction. Hence every endblock $B$ of every
	component $T$ of $G-Y$ is a $K_{k-1}$ and $E(\nonsep(B),Y)=\emptyset$.
	Let $x_{B}y_{B}\in E(\nonsep(B),Y)$.
	
	To each component $T$ of $G-Y$ we associate a set of edges $u(T)\subseteq E(W^{k}(T),Y)$
	as well as a \emph{type}, where $\type(T)\in\set{1,2a,2b,2c,3}$. Call a block $B$
	of $T$ \emph{saturated} if $\size{v,Y}\ne0$ for all $v\in\nonsep(B)$. For each
	component $T$ of $G-Y$, order the endblocks of $T$ as $B_{1},\ldots,B_{t}$ so
	that the saturated blocks come first. Define $u(T)$ and $\type(T)$ as follows:
	\begin{enumerate}
		\item $B_{1}$ is saturated. 
		
		\begin{enumerate}
			\item $t=1$ 
			
			\begin{itemize}
				\item put $u(T)=E(T,Y)$ and $\type(T)=2a$. 
			\end{itemize}
			\item $t\ge2$ 
			
			\begin{enumerate}
				\item $B_{2}$ is saturated 
				
				\begin{itemize}
					\item put $u(T)=E(\nonsep(B_{1}\cup B_{2}),Y)$ and $\type(T)=3$. 
				\end{itemize}
				\item $B_{2}$ is unsaturated 
				
				\begin{itemize}
					\item put $u(T)=E(\nonsep(B_{1}),Y)\cup\set{x_{B_{2}}y_{B_{2}}}$ and $\type(T)=2b$. 
				\end{itemize}
			\end{enumerate}
		\end{enumerate}
		\item Every endblock is unsaturated. 
		
		\begin{enumerate}
			\item $t=1$ 
			
			\begin{itemize}
				\item since $\delta(\B)\ge3$, there are three edges $e_{1},e_{2},e_{3}\in E(T,Y)$ with
				distinct ends in $Y$, put $u(T)=\set{e_{1},e_{2},e_{3}}$ and $\type(T)=1$. 
			\end{itemize}
			\item $t=2$ 
			
			\begin{enumerate}
				\item for some $i\in\irange{2}$, there are two edges $e_{1},e_{2}\in E(\nonsep(B_{i}),Y)$
				with distinct ends in $Y$ 
				
				\begin{itemize}
					\item put $u(T)=\set{e_{1},e_{2},x_{B_{3-i}}y_{B_{3-i}}}$ and $\type(T)=1$. 
				\end{itemize}
				\item otherwise, since $\delta(\B)\ge3$, there is an internal block $B_{0}=K_{k-1}$ with
				an edge $x_{B_{0}}y_{B_{0}}\in E(\nonsep(B),Y-y_{B_{1}}-y_{B_{2}})$ 
				
				\begin{enumerate}
					\item $B_{0}$ is saturated 
					
					\begin{itemize}
						\item put $u(T)=\set{x_{B_{1}}y_{B_{1}},x_{B_{2}}y_{B_{2}}}\cup E(\nonsep(B_{0}),Y)$ and
						$\type(T)=2c$. 
					\end{itemize}
					\item $B_{0}$ is unsaturated 
					
					\begin{itemize}
						\item put $u(T)=\set{x_{B_{1}}y_{B_{1}},x_{B_{2}}y_{B_{2}},x_{B_{0}}y_{B_{0}}}$ and $\type(T)=1$. 
					\end{itemize}
				\end{enumerate}
			\end{enumerate}
			\item $t\ge3$ 
			
			\begin{itemize}
				\item put $u(T)=\set{x_{B_{1}}y_{B_{1}},x_{B_{2}}y_{B_{2}},x_{B_{3}}y_{B_{3}}}$ and $\type(T)=1$. 
			\end{itemize}
		\end{enumerate}
	\end{enumerate}

	Every type other than type 1 results from a unique case of this definition. If $\type(T)\in\{2a,2b,2c\}$
	we also say $\type(T)=2$ (but type 2 vertices arise in three cases). If $\type(T)=i$
	then any $i$-set of independent edges of $u(T)$ either contains an edge ending
	in an unsaturated block or two edges ending in the same block.
	
	Let $\HH(T)=\{e=yT\in\HH:E_G(y,T)\cap u(T)\ne\emptyset\}$ and $h'(T)=|\HH(T)|$.  For $S \subseteq \B - Y$, let $h'(S) = \sum_{T \in S} h'(T)$.
	A component $T$ of $G-Y$ is \emph{heavy }if $\type(T) \leq h'(T)$; else $T$ is
	light. Define a function 
	\begin{eqnarray*}
		g:V(\B) & \rightarrow & \mathbb{N}\\
		v & \mapsto & \begin{cases}
			2-h(v) & \mbox{if}~v\in Y\\
			i-h'(T) & \mbox{if}~v=T,~T~\mbox{is light}~\mbox{and}~\type(T)=i\\
			0 & \mbox{if}~v=T~\mbox{and}~T~\mbox{is heavy}.
		\end{cases}
	\end{eqnarray*}
	Let $A$ be an incidence preference for $\B$ with $A(T)=\{yT\in E(\B):E_{G}(y,T)\cap u(T)\smallsetminus H\ne\emptyset\}$
	if $T$ is light, $A(T)=\emptyset$ if $T$ is heavy, and $A(y)=\{yT\in E(\B):E_{G}(y,T)\smallsetminus H\ne\emptyset\}$
	if $y\in Y$. We claim:
	\begin{equation}
		\mbox{There is an orientation \ensuremath{\D} of \ensuremath{\B} with \ensuremath{d_{\D}^{-}(v, A)\geq g(v)} for all \ensuremath{v\in V(\B)}.}\label{cl-2.8}
	\end{equation}

	By Lemma \ref{InOrientationsIncidencePreference}, it suffices to show every induced subgraph $\B'\subseteq\B$
	satisfies \[\eta \DefinedAs \sum_{v\in V(\B')} d_{\B}(v,A) - |A(\B')|-\sum_{v\in V(\B')} g(v) \geq 0.\] %HK Fixed typesetting
	Fix such a $\B'$. Let $Y'=Y\cap V(\B')$, $Q$ be the light vertices
	of type $1$ in $\B'$, $P$ be the light vertices of type $2$ in $\B'$
	and $R$ be the light vertices of type $3$ in $\B'$. Recall $\delta(\B)\geq 3$.
	For a light component $T$ of $G-Y$,
	\[
	d_{\B}(T,A)=\begin{cases}
	|\nonsep(B_{1}(T))|-2h'(T)=k-1-2h'(T), & \mbox{if}~\type(T)=2a\\
	|\nonsep(B_{1}(T))|-2h'(T)+1=k-1-2h'(T), & \mbox{if}~\type(T)=2b\\
	|\nonsep(B_0(T))|-2h'(T)+2=k-1-2h'(T), & \mbox{if}~\type(T)=2c\\
	|\nonsep(B_{1}(T))\cup\nonsep(B_{2}(T))|-2h'(T)=2k-4-2h'(T), & \mbox{if}~\type(T)=3.
	\end{cases}
	\]
	So, if $T\in P$ then $d_{\B}(T,A)=k-1-2h'(T)$ in $\B$. Thus
	\begin{eqnarray}
		\sum_{v\in V(\B')}d(v,A) & = & \sum_{v\in Y'}d(v,A)+\sum_{v\in Q\cup P\cup R}d(v,A)\label{2-l}\\
		\sum_{v\in Y'}d(v,A) & \geq & 3|Y'|-h(Y');\label{3-l}\\
		\sum_{v\in Q\cup P\cup R}d(v,A) & \geq & 3|Q|+(k-1)|P|+(2k-4)|R|-2h'(P \cup R);\label{4.5-l}\\
		|A(\B')| & \leq & \min\set{\sum_{v\in Y'}d(v,A),\sum_{v\in Q\cup P\cup R}d(v,A)};~\mbox{and}\label{4-l}\\
		\sum_{v\in V(\B')}g(v) & = & 2|Y'|+|Q|+2|P|+3|R|-h(Y')-h'(P)-h'(R).\label{5-l}
	\end{eqnarray}
	Using (\ref{2-l}, \ref{4-l}, \ref{3-l}, \ref{5-l}) yields 
	\begin{eqnarray}
		\eta & = & \sum_{v\in V(\B')}d(v,A)-|A(\B')|-\sum_{v\in V(\B')}g(v)\nonumber \\
		& \geq & |Y'|-|Q|-2|P|-3|R|+h'(P\cup R).\label{7-l}
	\end{eqnarray}
	Replacing \eqref{3-l} with \eqref{4.5-l} yields 
	\begin{equation}
		\eta\geq-2|Y'|+2|Q|+(k-3)|P|+(2k-7)|R|+h(Y')-h'(P\cup R).\label{8-l}
	\end{equation}
	Adding twice \eqref{7-l} to \eqref{8-l} yields
	
	\[
	3\eta\geq(k-7)|P|+2(k-6.5)|R|+h(Y')+h'(P\cup R).
	\]
	Since $k\geq7$, this implies $\eta\geq0$. So there exists an orientation $\mathcal{D}$
	satisfying \ref{cl-2.8}.
	
	Finally we use $\mathcal{D}$ to construct the subgraph $F\subseteq G$ needed
	in Lemma \ref{ConfigurationTypeTwoEuler}. For an edge $e=yT\in A(T)\cup\mathcal{H}(T)$,
	there is an edge $e'\in E_{G}(y,T)$ such that $e'\in u(T)$ if $e$ is light. If
	$e$ is heavy then there is another edge $e''\in E_{G}(y,T)$. Let 
	\[
	F=\{e':e=yT\in A(T)~\mbox{and}~yT\in E(\mathcal{D})\}\cup\{e':e=yT\in\mathcal{H}\}.
	\]

	We claim $F$ satisfies (4) of Lemma \ref{ConfigurationTypeTwoEuler}. Consider
	any component $T\in G-Y$; say $\type(T)=i$. Then there are at least $i$ edges
	$e'_{1}=x_{1}y_{1},\dots,e'_{i}=x_{i}y_{i}\in F$ with $y_{i}\in T$. Moreover, these
	edges are independent. 
	
	Suppose $i=1$. Then $x_{1}\in\bar{S}(B)$ for an unsaturated block $B\subseteq T$.
	As $B$ is unsaturated, there is a vertex $x\in\bar{S}(B)-x_{1}$ with no neighbor
	in $Y$. So $N[x_{1}]=N[x]$, and $e'_{1}$ and $x$ witness (4b).
	
	Suppose $i=2$. If $\type(T)=2a$, then $T$ has only one block $B_{1}$. So $\bar{S}(B_{1})=T$,
	and $e_{1}'$ and $e_{2}'$ witness (4a). If $\type(T)\in\{2b,2c\}$, then (4a) is
	satisfied if $x_{1}$ and $x_{2}$ are in the same block of $T$; else one of them
	ends in an unsaturated block, and (4b) is satisfied. 
	
	Finally, suppose $i=3$. Then (4a) is satisfied since two of $x_{1},x_{2},x_{3}$
	are in the same block. 
	
	Also, as each $y\in Y$ satisfies $d^{-}(y,A)\geq2-h(y)$, we have |$E(y,G-Y)\smallsetminus F|\geq2$.
	Thus $f(y)=d_{G}(y)-1\ge d_{G[Y]}(y)+d_{F}(y)+E(y,G-Y)\smallsetminus E(F)-1$. So
	(3) holds.
	\end{proof}

With a slightly simpler argument we get the following version with asymmetric degree condition on $\B$.  The point here is that this works for $k \ge 5$.  As we'll see in the next section, the consequence is that we trade a bit in our size bound for the proof to go through with $k \in \set{5,6}$.

\begin{lem}
	\label{MultipleHighConfigurationEulerLopsided} Let $k \ge 5$ and let $G$ be a graph with
	$Y\subseteq V(G)$ such that: 
	\begin{enumerate}
		\item $K_{k}\not\subseteq G$; and 
		\item the components of $G-Y$ are in $\T_{k}$; and 
		\item $d_{G}(v)\leq k-1$ for all $v\in V(G-Y)$; and 
		\item with $\B \DefinedAs \B_k(V(G-Y), Y)$ we have $d_{\B}(y) \ge 4$ for all $y \in Y$ and $d_{\B}(T) \ge 2$ for all components $T$ of $G-Y$.
	\end{enumerate}
	\noindent Then $G$ has an induced subgraph $G'$ that is $f$-AT where $f(y)=d_{G'}(y)-1$
	for $y\in Y$ and $f(v)=d_{G'}(v)$ for all $v\in V(G'-Y)$.\end{lem}
\begin{proof}
	\noindent Suppose not and pick a counterexample $G$ minimizing $\card{G}$. Note
	that $\size{w,Y}_G \le 1$ for every $w\in W^{k}(T)$, and if $\size{w,Y}_G = 1$
	then $w\in\nonsep(T)$; so if $y\in Y$ and $T$ is a component of $G-Y$ then $N(y)\cap W^{k}(T)\subseteq\nonsep(T)$.  
	By Lemma \ref{ConfigurationTypeOneEuler}, $\size{y,T}_G \le 2$ for each edge $yT$ of $\B$ since otherwise
	$G' = G[N_{\B}[y]]$ satisfies the conclusion of the lemma.
	Call an edge $yT$ of $\B$ \emph{heavy} if $\size{y,T}_G = 2$.
	Let $\HH$ be the set of heavy edges, and $H = \bigcup_{yT \in \HH} \setb{yx}{E(G)}{x \in V(T)}$. For $v\in S\subseteq V(\B)$,
	set $h(v)=\card{E_{\B}(v)\cap\HH}$ and $h(S)=\sum_{v\in S}h(v)$.
	By Lemma \ref{ConfigurationTypeOneEuler}, $h(y) \le 1$ for all $y \in Y$ since otherwise
	$G' = G[N_{\B}[y]]$ satisfies the conclusion of the lemma. 
	
	Suppose a component $T$ of $G-Y$ has an endblock $B$ with $B\ne K_{k-1}$ or $E(\nonsep(B),Y)=\emptyset$.
	Then $G'\DefinedAs G-\nonsep(B)$ still satisfies the hypotheses of the theorem since
	the degrees in $\B$ are not affected. Hence, by minimality of $\card{G}$, there
	is an induced subgraph $G''\subseteq G'$ that is $f$-AT where $f(y)=d_{G''}(y)-1$
	for $y\in Y$ and $f(v)=d_{G''}(v)$ for all $v\in V(G''-Y)$. But $G''$ is also
	an induced subgraph of $G$, a contradiction. Hence every endblock $B$ of every
	component $T$ of $G-Y$ is a $K_{k-1}$ and $E(\nonsep(B),Y)=\emptyset$.
	Let $x_{B}y_{B}\in E(\nonsep(B),Y)$.
	
	To each component $T$ of $G-Y$ we associate a set of edges $u(T)\subseteq E(W^{k}(T),Y)$
	as well as a \emph{type}, where $\type(T)\in\set{1,2a,2b,3}$. Call a block $B$
	of $T$ \emph{saturated} if $\size{v,Y}\ne0$ for all $v\in\nonsep(B)$. For each
	component $T$ of $G-Y$, order the endblocks of $T$ as $B_{1},\ldots,B_{t}$ so
	that the saturated blocks come first. Define $u(T)$ and $\type(T)$ as follows:
	
\begin{enumerate}
	\item $B_1$ is saturated.
	\begin{enumerate}
		\item $t=1$
		\begin{itemize}
			\item put $u(T) = E(T,Y)$ and $\type(T) = 2a$.
		\end{itemize}
		\item $t \ge 2$
		\begin{enumerate}
			\item $B_2$ is saturated
			\begin{itemize}
				\item put $u(T) = E(\nonsep(B_1 \cup B_2),Y)$ and $\type(T) = 3$.
			\end{itemize}
			\item $B_2$ is unsaturated
			\begin{itemize}
				\item put $u(T) = E(\nonsep(B_1),Y) \cup \set{x_{B_2}y_{B_2}}$ and $\type(T) = 2b$.
			\end{itemize}
		\end{enumerate}
	\end{enumerate}
	\item $B_1$ is unsaturated.
	\begin{enumerate}
		\item $t=1$
		\begin{itemize}
			\item since $\delta(\B) \ge 2$, there are two edges $e_1, e_2 \in E(T,Y)$ with distinct ends in $Y$, put $u(T) = \set{e_1, e_2}$ and $\type(T) = 1$.
		\end{itemize}
		\item $t \ge 2$
		\begin{itemize}
			\item put $u(T) = \set{x_{B_1}y_{B_1},x_{B_2}y_{B_2}}$ and $\type(T) = 1$.
		\end{itemize}
	\end{enumerate}
\end{enumerate}

	Every type other than type 1 results from a unique case of this definition. If $\type(T)\in\{2a,2b\}$
	we also say $\type(T)=2$ (but type 2 vertices arise in three cases). If $\type(T)=i$
	then any $i$-set of independent edges of $u(T)$ either contains an edge ending
	in an unsaturated block or two edges ending in the same block.
	
	Let $\HH(T)=\{e=yT\in\HH:E_G(y,T)\cap u(T)\ne\emptyset\}$ and $h'(T)=|\HH(T)|$.  For $S \subseteq \B - Y$, let $h'(S) = \sum_{T \in S} h'(T)$.
	A component $T$ of $G-Y$ is \emph{heavy }if $\type(T) \leq h'(T)$; else $T$ is
	light. Define a function 
	\begin{eqnarray*}
		g:V(\B) & \rightarrow & \mathbb{N}\\
		v & \mapsto & \begin{cases}
			2-h(v) & \mbox{if}~v\in Y\\
			i-h'(T) & \mbox{if}~v=T,~T~\mbox{is light}~\mbox{and}~\type(T)=i\\
			0 & \mbox{if}~v=T~\mbox{and}~T~\mbox{is heavy}.
		\end{cases}
	\end{eqnarray*}
	Let $A$ be an incidence preference for $\B$ with $A(T)=\{yT\in E(\B):E_{G}(y,T)\cap u(T)\smallsetminus H\ne\emptyset\}$
	if $T$ is light, $A(T)=\emptyset$ if $T$ is heavy, and $A(y)=\{yT\in E(\B):E_{G}(y,T)\smallsetminus H\ne\emptyset\}$
	if $y\in Y$. We claim:
	\begin{equation}
	\mbox{There is an orientation \ensuremath{\D} of \ensuremath{\B} with \ensuremath{d_{\D}^{-}(v, A)\geq g(v)} for all \ensuremath{v\in V(\B)}.}\label{cl-2.8}
	\end{equation}

	By Lemma \ref{InOrientationsIncidencePreference}, it suffices to show every induced subgraph $\B'\subseteq\B$
	satisfies \[\eta \DefinedAs \sum_{v\in V(\B')} d_{\B}(v,A) - |A(\B')|-\sum_{v\in V(\B')} g(v) \geq 0.\]
	Fix such a $\B'$. Let $Y'=Y\cap V(\B')$, $Q$ be the light vertices
	of type $1$ in $\B'$, $P$ be the light vertices of type $2$ in $\B'$
	and $R$ be the light vertices of type $3$ in $\B'$. Recall $d_{\B}(y) \ge 4$ for all $y \in Y$ and $d_{\B}(T) \ge 2$ for all components $T$ of $G-Y$.
	For a light component $T$ of $G-Y$,
	\[
	d_{\B}(T,A)=\begin{cases}
	|\nonsep(B_{1}(T))|-2h'(T)=k-1-2h'(T), & \mbox{if}~\type(T)=2a\\
	|\nonsep(B_{1}(T))|-2h'(T)+1=k-1-2h'(T), & \mbox{if}~\type(T)=2b\\
	|\nonsep(B_{1}(T))\cup\nonsep(B_{2}(T))|-2h'(T)=2k-4-2h'(T), & \mbox{if}~\type(T)=3.
	\end{cases}
	\]
	So, if $T\in P$ then $d_{\B}(T,A)=k-1-2h'(T)$ in $\B$. Thus
	\begin{eqnarray}
	\sum_{v\in V(\B')}d(v,A) & = & \sum_{v\in Y'}d(v,A)+\sum_{v\in Q\cup P\cup R}d(v,A)\label{2-l}\\
	\sum_{v\in Y'}d(v,A) & \geq & 4|Y'|-h(Y');\label{3-l}\\
	\sum_{v\in Q\cup P\cup R}d(v,A) & \geq & 2|Q|+(k-1)|P|+(2k-4)|R|-2h'(P \cup R);\label{4.5-l}\\
	|A(\B')| & \leq & \min\set{\sum_{v\in Y'}d(v,A),\sum_{v\in Q\cup P\cup R}d(v,A)};~\mbox{and}\label{4-l}\\
	\sum_{v\in V(\B')}g(v) & = & 2|Y'|+|Q|+2|P|+3|R|-h(Y')-h'(P)-h'(R).\label{5-l}
	\end{eqnarray}
	Using (\ref{2-l}, \ref{4-l}, \ref{3-l}, \ref{5-l}) yields 
	\begin{eqnarray}
	\eta & = & \sum_{v\in V(\B')}d(v,A)-|A(\B')|-\sum_{v\in V(\B')}g(v)\nonumber \\
	& \geq & 2|Y'|-|Q|-2|P|-3|R|+h'(P\cup R).\label{7-l}
	\end{eqnarray}
	Replacing \eqref{3-l} with \eqref{4.5-l} yields 
	\begin{equation}
	\eta\geq-2|Y'|+|Q|+(k-3)|P|+(2k-7)|R|+h(Y')-h'(P\cup R).\label{8-l}
	\end{equation}
	Adding \eqref{7-l} to \eqref{8-l} yields
	
	\[
	2\eta\geq(k-5)|P|+2(k-5)|R|+h(Y').
	\]
	Since $k\geq5$, this implies $\eta\geq0$. So there exists an orientation $\mathcal{D}$
	satisfying \ref{cl-2.8}.
	
	Finally we use $\mathcal{D}$ to construct the subgraph $F\subseteq G$ needed
	in Lemma \ref{ConfigurationTypeTwoEuler}. For an edge $e=yT\in A(T)\cup\mathcal{H}(T)$,
	there is an edge $e'\in E_{G}(y,T)$ such that $e'\in u(T)$ if $e$ is light. If
	$e$ is heavy then there is another edge $e''\in E_{G}(y,T)$. Let 
	\[
	F=\{e':e=yT\in A(T)~\mbox{and}~yT\in E(\mathcal{D})\}\cup\{e':e=yT\in\mathcal{H}\}.
	\]

	We claim $F$ satisfies (4) of Lemma \ref{ConfigurationTypeTwoEuler}. Consider
	any component $T\in G-Y$; say $\type(T)=i$. Then there are at least $i$ edges
	$e'_{1}=x_{1}y_{1},\dots,e'_{i}=x_{i}y_{i}\in F$ with $y_{i}\in T$. Moreover, these
	edges are independent. 
	
	Suppose $i=1$. Then $x_{1}\in\bar{S}(B)$ for an unsaturated block $B\subseteq T$.
	As $B$ is unsaturated, there is a vertex $x\in\bar{S}(B)-x_{1}$ with no neighbor
	in $Y$. So $N[x_{1}]=N[x]$, and $e'_{1}$ and $x$ witness (4b).
	
	Suppose $i=2$. If $\type(T)=2a$, then $T$ has only one block $B_{1}$. So $\bar{S}(B_{1})=T$,
	and $e_{1}'$ and $e_{2}'$ witness (4a). If $\type(T) = 2b$, then (4a) is
	satisfied if $x_{1}$ and $x_{2}$ are in the same block of $T$; else one of them
	ends in an unsaturated block, and (4b) is satisfied. 
	
	Finally, suppose $i=3$. Then (4a) is satisfied since two of $x_{1},x_{2},x_{3}$
	are in the same block. 
	
	Also, as each $y\in Y$ satisfies $d^{-}(y,A)\geq2-h(y)$, we have |$E(y,G-Y)\smallsetminus F|\geq2$.
	Thus $f(y)=d_{G}(y)-1\ge d_{G[Y]}(y)+d_{F}(y)+E(y,G-Y)\smallsetminus E(F)-1$. So
	(3) holds.
\end{proof}

\section{Main theorem: AT-irreducible graphs have many edges}
The rest of the proof is basically taken verbatim from \cite{kostochkastiebitzedgesincriticalgraph}. We need the following definitions:
\begin{align*}
\L_k(G) &\DefinedAs G\brackets{x \in V(G) \mid d_G(x) < k},\\
\HH_k(G) &\DefinedAs G\brackets{x \in V(G) \mid d_G(x) \ge k},\\
\sigma_k(G) &\DefinedAs \parens{k-2 + \frac{2}{k-1}}\card{\L_k(G)} - 2\size{\L_k(G)},\\
\tau_{k,c}(G) &\DefinedAs 2\size{\HH_k(G)} + \parens{k-c - \frac{2}{k-1}}\sum_{y \in V(\HH_k(G))} \parens{d_G(y) - k},\\
\alpha_k &\DefinedAs \frac12 - \frac{1}{(k-1)(k-2)},\\
q_k(G) &\DefinedAs \alpha_k\sum_{v \in V(G) \setminus W^k(G)} \parens{k-1 - d_G(v)},\\
g_k(n, c) &\DefinedAs \parens{k-1 + \frac{k-3}{(k-c)(k-1) + k-3}}n.\\
\end{align*}

\noindent As proved in \cite{kostochkastiebitzedgesincriticalgraph}, a computation gives the following.
\begin{lem}\label{SigmaTauBoundEuler}
Let $G$ be a graph with $\delta \DefinedAs \delta(G) \ge 3$ and $0 \leq c \leq \delta + 1 - \frac{2}{\delta}$.  If $\sigma_{\delta + 1}(G) + \tau_{\delta + 1, c}(G) \ge c\card{\HH_{\delta + 1}(G)}$, then $2\size{G} \ge g_{\delta + 1}(\card{G}, c)$.
\end{lem}

\noindent We need the following degeneracy lemma.
\begin{lem}\label{DegenerateEuler}
Let $G$ be a graph and $\func{f}{V(G)}{\IN}$.  If $\size{G} > \sum_{v \in V(G)} f(v)$, then $G$ has an induced subgraph $H$ such that $d_H(v) > f(v)$ for each $v \in V(H)$.
\end{lem}
\begin{proof}
Suppose not and choose a counterexample $G$ minimizing $\card{G}$. Then $\card{G} \ge 3$ and we have $x \in V(G)$ with $d_G(x) \leq f(x)$. But now $\size{G-x} > \sum_{v \in V(G-x)} f(v)$, contradicting minimality of $\card{G}$.
\end{proof}

We'll also need the following consequence of Lemma 2.3 in \cite{kostochkastiebitzedgesincriticalgraph} giving a lower bound on $\sigma_k(T)$ for $T \in \T_k$. Lemma 2.3 in \cite{kostochkastiebitzedgesincriticalgraph} is only proved for $k \ge 6$, but we need our lemma to work for $k=5$ as well, so we prove that here. Notice that when $T \in \T_k$, we have $\L_k(T) = T$.  We also use the following simple fact (Lemma 2.1(b) in \cite{kostochkastiebitzedgesincriticalgraph}): if $B$ is an endblock of $T \in \T_k$ and $x$ is the unique cutvertex of $T$ in $V(B)$, then $\sigma_k(T) = \sigma_k(T - (B-x)) + \sigma_k(B) - (k-2 + \frac{2}{k-1})$.

\begin{lem}\label{SigmaBoundEuler}
	Let $k \ge 5$ and $T \in \T_k$. If $K_{k-1} \subseteq T$, then $\sigma_k(T) \ge 2 + q_k(T)$; otherwise $\sigma_k(T) \ge 2 - \alpha_k + q_k(T)$.
\end{lem}
\begin{proof}
	Suppose the lemma is false and choose a counterexample $T \in \T_k$ minimizing $\card{T}$.  By Lemma 2.3 in \cite{kostochkastiebitzedgesincriticalgraph}, we have $k=5$.  Then $\alpha_5 = \frac{5}{12}$ and $2-\alpha_5 = \frac{19}{12}$.  Also, $\sigma_5(T) = \frac72|T| - 2\size{T}$ and $q_5(T) = \frac{5}{12}\sum_{v \in V(G) \setminus W^5(G)} \parens{4 - d_T(v)}$.	Suppose $T$ has only one block.  First, suppose $T = K_t$ for $t \in \set{2,3}$. Then $\sigma_5(T) - q_5(T) = \frac72t - t(t-1) - \frac{5}{12}t(5-t)\ge \frac{19}{12}$, a contradiction. Also, $T \ne K_4$ since then $\sigma_5(T) = 2$ and $q_5(T) = 0$.  If $T$ is an odd cycle of length $\ell$, then $\sigma_5(T) - q_5(T) = \frac32\ell - \frac{5}{12}(2\ell) = \frac23\ell \ge \frac{19}{12}$, a contradiction.  
	
	So, $T$ must have at least two blocks.  For an endblock $B$ of $T$, let $x_B$ be the unique cut vertex of $T$ in $V(B)$.  Consider $T_B \DefinedAs T - (B-x)$.  Clearly (and by Lemma 2.1(b) in \cite{kostochkastiebitzedgesincriticalgraph}), %HK I think this is too easy to just cite
	 we have $\sigma_5(T) = \sigma_5(T_B) + \sigma_5(B) - \frac72$.	Suppose $B$ has an endblock $B \ne K_4$. Since $B \ne K_4$, any $K_4$ in $T$ is in $T_B$, %HK T_B, not T'
	 so minimality of $|T|$ yields $\sigma_5(T_B) \ge 2 + q_5(T_B)$ if $K_4 \subseteq T$ and $\sigma_5(T_B) \ge 2 - \alpha_5 + q_5(T_B)$ otherwise.  Since $T$ is a counterexample, we must have $\sigma_5(B) - \frac72 + q_5(T_B) < q_5(T)$.  Since $B$ is regular, this gives
	\begin{align*}
	\frac72|B| - 2\size{B} - \frac72 &= \sigma_5(B) - \frac72 \\
	&< \alpha_5 \parens{-d_{B}(x_B)  + \sum_{v \in V(B - x_B)} 4 - d_B(v)}\\
	&= \frac{5}{12}\parens{-\Delta(B) + (|B|- 1)(4 - \Delta(B))}.
	\end{align*} Therefore,
	\[\frac72(|B| - 1) - |B|\Delta(B) < \frac{5}{12}\parens{-\Delta(B) + (|B|- 1)(4 - \Delta(B))}.\]  This simplifies to the following which is a contradiction since $\Delta(B) \in \set{1,2}$:
	\[\Delta(B) > \frac{22}{7}\parens{1 - \frac{1}{|B|}}.\]
	
	Therefore, every endblock of $T$ is $K_4$.  Choose an endblock $B = K_4$ of $T$.  Then the other block containing $x_B$ is a $K_2$, let $y$ be the other vertex in this $K_2$.  Consider $T' = T - B$.  Then $K_4 \subseteq T'$ since $T$ had another endblock which must be $K_4$.  By minimality of $|T|$, we conclude $\sigma_5(T') \ge 2 + q_5(T')$.  Since $T$ has $4$ more vertices and $7$ more edges than $T'$, we have $\sigma_5(T) = \sigma_5(T') + 4\frac72 - (2)(7) = \sigma_5(T')$.  Also, $q_5(T') = q_5(T)$ if $y$ is in a $K_{k-1}$ and $q_5(T') = q_5(T) + \alpha_5$ otherwise (since all the vertices in $B$ are in a $K_{k-1}$, they do not contribute to $q_5(T)$).  Hence $\sigma_5(T) = \sigma_5(T') \ge 2 + q_5(T') \ge 2 + q_5(T)$, a contradiction.
\end{proof}

\noindent We are now ready to prove the main theorem.

\begin{thm}\label{EdgeBoundEuler}
	If $G$ is an AT-irreducible graph with $\delta(G) \ge 4$ and $\omega(G) \le \delta(G)$, then $2\size{G} \ge g_{\delta(G)+1}(\card{G}, c)$ where $c \DefinedAs (\delta(G)-2)\alpha_{\delta(G) + 1}$ when $\delta(G) \ge 6$ and $c \DefinedAs (\delta(G)-3)\alpha_{\delta(G) + 1}$ when $\delta(G) \in \set{4,5}$.
\end{thm}
\begin{proof}
Put $k \DefinedAs \delta(G) + 1$, $\L \DefinedAs \L_k(G)$ and $\HH \DefinedAs \HH_k(G)$. Plainly, $c \le \delta(G) + 1 - \frac{2}{\delta(G)}$. So, using Lemma \ref{SigmaTauBoundEuler}, we just need to show that $\sigma_k(G) + \tau_{k, c}(G) \ge c\card{\HH}$.  Put $W \DefinedAs W^k(\L)$, $L' \DefinedAs V(\L) \setminus W$ and $H' \DefinedAs \setb{v}{V(\HH)}{d_G(v) = k}$.   For $y \in V(\HH)$, put 
\[\tau_{k,c}(y) \DefinedAs d_{\HH}(y) + \parens{k-c + \frac{2}{k-1}}(d_G(y) - k).\]  We have 
\begin{align*}
\tau_{k,c}(G) &= \sum_{y \in V(\HH)} \tau_{k,c}(y)\\
			  &\ge \sum_{y \in H'} d_{\HH}(y) + \sum_{y \in V(\HH) \setminus H'} \parens{d_{\HH}(y) + k-c + \frac{2}{k-1}}\\
			  &\ge \sum_{y \in H'} d_{\HH}(y) + \parens{k-c + \frac{2}{k-1}}\card{\HH - H'}\\
              &\ge \sum_{y \in H'} d_{\HH}(y) + c\card{\HH - H'},
\end{align*}
where the last inequality follows since $c \le (k-3)\alpha_k = (k-3)\parens{\frac{1}{2} - \frac{1}{(k-1)(k-2)}} \le \frac{k}{2}$.  Therefore, it will be sufficient to prove that $S \DefinedAs \sigma_k(G) + \sum_{y \in H'} d_{\HH}(y) \ge c\card{H'}$.

Let $\D$ be the components of $\L$ containing $K_{k-1}$ and $\CC$ the components of $\L$ not containing $K_{k-1}$.  Then $\D \cup \CC \subseteq \T_k$ for otherwise some $T \in \D \cup \CC$ is $d_0$-AT and hence $f_T$-AT and $G$ is AT-reducible.  By Lemma \ref{SigmaBoundEuler}, we have $\sigma_k(T) \ge 2 + q_k(T)$ for if $T \in \D$ and $\sigma_k(T) \ge 2 - \alpha_k + q_k(T)$ if $T \in \CC$.  Hence, we have $\sigma_k(G) = \sum_{T \in \D} \sigma_k(T) + \sum_{T \in \CC} \sigma_k(T) \ge 2\card{\D} + (2-\alpha_k)\card{\CC} + \alpha_k\sum_{v \in L'} \parens{k-1 - d_{\L}(v)}$.

Now we define an auxiliary bipartite graph $F$ with parts $A$ and $B$ where:
\begin{enumerate}
\item  $B = H'$ and $A$ is the disjoint union of the following sets
$A_1, A_2$ and $A_3$,
\item $A_1 = \D$ and each $T \in \D$ is adjacent to all $y \in H'$
where $N(y) \cap W^k(T) \ne \emptyset$,
\item For each $v \in L'$, let $A_2(v)$ be a set of $\card{N(v) \cap
H'}$ vertices connected to $N(v) \cap H'$ by a matching in $F$.  Let
$A_2$ be the disjoint union of the $A_2(v)$ for $v \in L'$,
\item For each $y \in H'$, let $A_3(y)$ be a set of $d_{\HH}(y)$ vertices
which are all joined to $y$ in $F$.  Let $A_3$ be the disjoint union
of the $A_3(y)$ for $y \in H'$.
\end{enumerate}

\noindent \case{1}{$\delta \ge 6$.}
\smallskip

Define $\func{f}{V(F)}{\IN}$ by $f(v) = 1$ for all $v \in A_2 \cup A_3$ and $f(v) = 2$ for all $v \in B \cup A_1$.  First, suppose $\size{F} > \sum_{v \in V(F)} f(v)$.  Then by Lemma \ref{DegenerateEuler}, $F$ has an induced subgraph $Q$ such that $d_Q(v) > f(v)$ for each $v \in V(Q)$.  In particular, $V(Q) \subseteq B \cup A_1$ and $\delta(Q) \ge 3$.  Put $Y \DefinedAs B \cap V(Q)$ and let $X$ be $\bigcup_{T \in V(Q) \cap A_1} V(T)$. Now $H \DefinedAs G[X \cup Y]$ satisfies the hypotheses of Lemma \ref{MultipleHighConfigurationEuler}, so $H$ has an induced subgraph $G'$ that is $f$-AT where $f(y) = d_{G'}(y) - 1$ for $y \in Y$ and $f(v) = d_{G'}(v)$ for $v \in X$.  Since $Y \subseteq H'$ and $X \subseteq \L$, we have $f(v) = \delta(G) + d_{G'}(v) - d_G(v)$ for all $v \in V(G')$.  Hence, $G$ is AT-reducible to $G'$, a contradiction.

Therefore $\size{F} \leq \sum_{v \in V(F)} f(v) = 2(\card{H'} + \card{\D}) + \card{A_2} + 
\card{A_3}$. By Lemma \ref{ConfigurationTypeOneEuler}, for each $y \in B$ we have $d_F(y) \ge k-1$.  Hence $\size{F} \ge (k-1)\card{H'}$.  This gives $(k-3)\card{H'} \leq 2\card{\D} + \card{A_2} + 
\card{A_3}$.  By our above estimate we have $S \ge 2\card{\D} + \alpha_k\sum_{v \in L'} \parens{k-1 - d_{\L}(v)}  + \sum_{y \in H'} d_{\HH}(y) = 2\card{\D} + \alpha_k\card{A_2} + \card{A_3} \ge \alpha_k(2\card{\D} + \card{A_2} + \card{A_3})$.  Hence $S \ge \alpha_k(k-3)\card{H'}$.  Thus our desired bound holds by Lemma \ref{SigmaTauBoundEuler}.
\bigskip

\noindent \case{2}{$\delta \in \set{4,5}$.}
\smallskip

Define $\func{f}{V(F)}{\IN}$ by $f(v) = 1$ for all $v \in A_1 \cup A_2 \cup A_3$ and $f(v) = 3$ for all $v \in B$.  First, suppose $\size{F} > \sum_{v \in V(F)} f(v)$.  Then by Lemma \ref{DegenerateEuler}, $F$ has an induced subgraph $Q$ such that $d_Q(v) > f(v)$ for each $v \in V(Q)$.  In particular, $V(Q) \subseteq B \cup A_1$ and $d_Q(v) \ge 4$ for $v \in B \cap V(Q)$ and $d_Q(v) \ge 2$ for $v \in A_1 \cap V(Q)$.  Put $Y \DefinedAs B \cap V(Q)$ and let $X$ be $\bigcup_{T \in V(Q) \cap A_1} V(T)$. Now $H \DefinedAs G[X \cup Y]$ satisfies the hypotheses of Lemma \ref{MultipleHighConfigurationEulerLopsided}, so $H$ has an induced subgraph $G'$ that is $f$-AT where $f(y) = d_{G'}(y) - 1$ for $y \in Y$ and $f(v) = d_{G'}(v)$ for $v \in X$.  Since $Y \subseteq H'$ and $X \subseteq \L$, we have $f(v) = \delta(G) + d_{G'}(v) - d_G(v)$ for all $v \in V(G')$.  Hence, $G$ is AT-reducible to $G'$, a contradiction.

Therefore $\size{F} \leq \sum_{v \in V(F)} f(v) = 3\card{H'} + \card{\D} + \card{A_2} + 
\card{A_3}$. By Lemma \ref{ConfigurationTypeOneEuler}, for each $y \in B$ we have $d_F(y) \ge k-1$.  Hence $\size{F} \ge (k-1)\card{H'}$.  This gives $(k-4)\card{H'} \leq \card{\D} + \card{A_2} + 
\card{A_3}$.  By our above estimate we have $S \ge 2\card{\D} + \alpha_k\sum_{v \in L'} \parens{k-1 - d_{\L}(v)}  + \sum_{y \in H'} d_{\HH}(y) = 2\card{\D} + \alpha_k\card{A_2} + \card{A_3} \ge \alpha_k(\card{\D} + \card{A_2} + \card{A_3})$.  Hence $S \ge \alpha_k(k-4)\card{H'}$.  Thus our desired bound holds by Lemma \ref{SigmaTauBoundEuler}.
\end{proof}

We note a corollary of the above proof that will be useful in a later paper.  When $\HH_k(G)$ is edgeless, $A_3$ is empty and $S = \sigma_k(G)$.  Also from the proof, we have $\sigma_k(G) \ge 2\card{\D} + (2-\alpha_k)\card{\CC} + \alpha_k\sum_{v \in L'} \parens{k-1 - d_{\L}(v)} \ge  \alpha_k(2\card{\D} + \card{A_2}) + 2(1-\alpha_k)\card{\D} + (2-\alpha_k)\card{\CC}$.   We write $c(G)$ for the number of components of $G$.  Since $(2-\alpha_k) \ge 2(1-\alpha_k)$, we have $\sigma_k(G) \ge (k-3)\alpha_k\card{\HH_k(G)} + 2(1-\alpha_k)c(\L(G))$.

\begin{cor}\label{SigmaCorollary}
If $G$ is an AT-irreducible graph with $\delta \DefinedAs \delta(G) \ge 6$ and $\omega(G) \le \delta$ such that $\HH_{\delta+1}(G)$ is edgeless, then $\sigma_{\delta + 1}(G) \ge (\delta-2)\alpha_{\delta + 1}\card{\HH_{\delta+1}(G)} + 2(1-\alpha_{\delta + 1})c(\L(G))$.
\end{cor}

\section{Corollaries: Critical graphs have many edges}
\begin{cor}\label{EdgeBound}
For $k \ge 5$ and $G \ne K_k$ a $k$-list-critical graph, we have $2\size{G} \ge g_k(\card{G}, c)$ where $c \DefinedAs (k-3)\alpha_k$ when $k \ge 7$ and $c \DefinedAs (k-4)\alpha_k$ when $k \in \set{5,6}$.
\end{cor}
\begin{proof}
Let $L$ be a $(k-1)$-assignment such that $G$ is $L$-critical.  Since $G$ is $L$-critical, we have $\delta(G) \ge k - 1 \ge 5$. If $\delta(G) \ge k$, then $2\size{G} \ge k\card{G} \ge g_k(\card{G}, k)$ and we are done.  Hence we may assume that $\delta(G) = k-1$.  Since $G \ne K_k$ and $G$ is $L$-critical, we have $K_{\delta(G) + 1} \not \subseteq G$.  By Lemma \ref{ChoosingIrreducible}, $G$ is AT-irreducible, so Lemma \ref{EdgeBoundEuler} proves the corollary.
\end{proof}

Note that applying Lemma \ref{CutLemma} where $H$ has a single vertex shows that $\delta(G) \ge k - 1$ for an online $k$-list-critical graph.

\begin{cor}\label{EdgeBoundOnline}
For $k \ge 5$ and $G \ne K_k$ an online $k$-list-critical graph, we have  $2\size{G} \ge g_k(\card{G}, c)$ where $c \DefinedAs (k-3)\alpha_k$ when $k \ge 7$ and $c \DefinedAs (k-4)\alpha_k$ when $k \in \set{5,6}$.
\end{cor}
\begin{proof}
Since $G$ is online $k$-list-critical, we have $\delta(G) \ge k - 1 \ge 5$.  If $\delta(G) \ge k$, then $2\size{G} \ge k\card{G} \ge g_k(\card{G}, k)$ and we are done.  Hence we may assume that $\delta(G) = k-1$. Since $G \ne K_k$ and $G$ is online $k$-list-critical, we have $K_{\delta(G) + 1} \not \subseteq G$.   By Lemma \ref{OnlineChoosingIrreducible}, $G$ is AT-irreducible, so Lemma \ref{EdgeBoundEuler} proves the corollary.
\end{proof}

Note that applying Lemma \ref{CutLemmaAT} where $H$ has a single vertex shows that $\delta(G) \ge k - 1$ for a $k$-AT-critical graph $G$.

\begin{cor}\label{EdgeBoundAT}
For $k \ge 5$ and $G \ne K_k$ a $k$-AT-critical graph, we have $2\size{G} \ge g_k(\card{G}, c)$ where $c \DefinedAs (k-3)\alpha_k$ when $k \ge 7$ and $c \DefinedAs (k-4)\alpha_k$ when $k \in \set{5,6}$.
\end{cor}
\begin{proof}
Since $G$ is $k$-AT-critical, we have $\delta(G) \ge k - 1 \ge 5$.  If $\delta(G) \ge k$, then $2\size{G} \ge k\card{G} \ge g_k(\card{G}, k)$ and we are done.  Hence we may assume that $\delta(G) = k-1$. Since $G \ne K_k$ and $G$ is $k$-AT-critical, we have $K_{\delta(G) + 1} \not \subseteq G$.   By Lemma \ref{ATNumberIrreducible}, $G$ is AT-irreducible, so Lemma \ref{EdgeBoundEuler} proves the corollary.
\end{proof}

\bibliographystyle{amsplain}
\bibliography{GraphColoring1}%HK fixed [8,9,10]
\end{document}